\documentclass[11pt]{article}
\usepackage{latexcad}
\usepackage[active]{srcltx}
\usepackage[centertags]{amsmath}
\usepackage{amsfonts}
\usepackage{amssymb}
\usepackage{amsthm}
\newcommand{\Title}[1]{{\Large \bf \begin{center} {#1}
  \end{center}}\vspace*{1mm}}
\newcommand{\Author}[1]{{\bf \begin{center} {#1}
  \end{center}}\vspace*{1mm}}
\newtheorem{thm}{Theorem} 
\newtheorem{cor}[thm]{Corollary}
\newtheorem{lem}[thm]{Lemma}

\begin{document}
\Title{\Large Free Minor Closed Classes and the Kuratowski
theorem. } \Author{ Dainis ZEPS \footnote{This research is
supported by the grant 97.0247 of Latvian Council of Science.}
\footnote{Author's address: Institute of Mathematics and Computer
Science, University of Latvia, 29 Rainis blvd., Riga, Latvia.
{dainize@cclu.lv}} }
\begin{abstract}
Free-minor closed classes \cite{kra 94} and free-planar graphs
\cite{ze 90} are considered. Versions of Kuratowski-like theorem
for free-planar graphs  and Kuratowski theorem for planar graphs
are considered.
\end{abstract}

 We are using usual definitions of the graph theory \cite{gt xx}.  Considering graph topologically and Kuratowski theorem, we use the notion of minor following the theory of Robertson and Seymour\cite{kra 94}. We say, that a graph $G$ is a minor of a graph $H$, denoting it by $G \prec H$,  if  $G$ can be obtained from $H$ by edge contractions from a subgraph of $H$, i.e. $G$ can be obtained by vertex deletions, edge deletions and edge contractions from $H$.

A class of graphs $A$ is minor closed, if  from $G \in A$ and $ H \prec G$  follow that $ H \in A$.

The set of forbidden minors of a class $A$ is denoted by  $F(A)$
which is equal to  $\lfloor \{G \mid G \not\in A\}\rfloor$, where
$\lfloor  B  \rfloor$ contains only minimal minors of $B$:
$\lfloor B \rfloor \triangleq \{G \mid H \in B \wedge H \prec G
\Rightarrow H \cong G\}$.

$N_{\circ}(B)$ denotes the minor closed class with $B$ as its set
of forbidden minors, i.e. $N_{\circ}(B) \triangleq \{G \mid
\forall H \in B: H \not\prec G\}$. In other words, we may say,
that $N_{\circ}(B)$ is a minor closed class generated by its
forbidden minors in $B$. For example, $N_{\circ}( K_5,K_{3,3})$ is
the class of planar graphs, as it is stated by Kuratowski theorem.

Another interesting example is {\em free-planar graphs} \cite{ze 90}. A planar graph is called free-planar, if after adding an arbitrary edge it remains to be planar. In \cite{ze 90} without a proof is acclaimed, that  the class of free-planar graphs is equal to $N_{\circ}( K_5^-,K_{3,3}^-)$, and  its characterization in terms of the permitted 3-connected components is given. In this paper we give a proof of this characterization.

In \cite{kra 94} a generalization of the notion of free-planar graphs is suggested. We denote by  $Free(A)$ the class of graphs that consist of all graphs which should belong to $A$ after adding an arbitrary edge to them.
It is easy to see, that, if $A$ is minor closed, then $Free(A)$ is minor closed too \cite{kra 94}. Because of this we use to say, that $Free(A)$ is {\em free-minor-closed-class} for a minor closed class $A$.

In \cite{kra 94} Kratochv\'{\i}l proved a theorem: $$F(Free(A)) =
\lfloor F(A)^- \cup F(A)^{\odot}\rfloor ,$$ where $B^- \triangleq
\{G-e \mid G \in  B, e \in E(G)\}$ and $B^{\odot} \triangleq \{H
\mid H \cong G \odot v, G \in B, v \in V(G)\}$ and operation $
\odot $ [in its application $G \odot v$] denotes a non unique
splitting of vertex $v$ in $G$, which is the opposite operation to
edge addition and its contraction [in result giving vertex $v$].

We may formulate the unproved statement of \cite{ze 90} as a theorem for class of planar graphs $Planar$:

\begin{thm}\label{fpg} $Free(Planar) = N_{\circ}( K_5^-,K_{3,3}^-)$.
\end{thm}

It is convenient to call the graphs $ K_5^-,K_{3,3}^-$ -- reduced Kuratowski subgraphs (or minors or graphs).

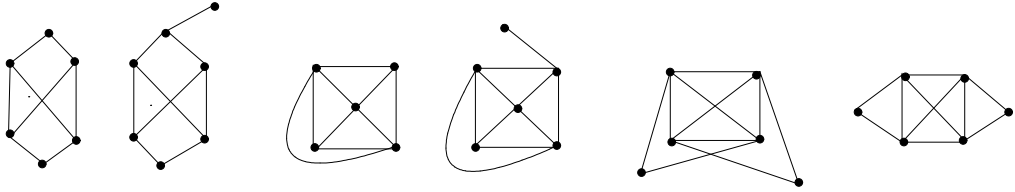
\placedrawing{FIGURE0.LP}{Graphs received applying the theorem of Kratochv\'{\i}l to Kuratowski graphs.}{fig0}

Now, direct application of  the theorem of Kratochv\'{\i}l gives the proof of  theorem \ref{fpg}, that has been   already shown in  \cite{kra 94}. All possible graphs obtainable following the theorem are in fig. \ref{fig0}.

In \cite{kra 94} Kratochv\'{\i}l suggested to prove Kuratowski's theorem from its weaker version for free-planar graphs. We do this here in two ways. One way -- first specifying the class generated from reduced Kuratowski minors and then showing that it coincides with  the class of free-planar graphs and then proving Kuratowski theorem itself. Second way -- we prove Kuratowski theorem for free planar graphs directly, showing that with slight alteration this proof fits  for a complete class of planar graphs too.

Let us set  $FP=N_{\circ}( K_5^-,K_{3,3}^-)$ and start with proving, that graphs belonging to the class $FP$ are free-planar, i.e. an extra edge does not make them nonplanar. Here we should explain how we are going to use Kuratowski theorem during the time we prove it. From a fact that $G$ has a Kuratowski graph as minor we conclude that it is non-planar, i.e. we use the weak direction of Kuratowski theorem. Otherwise we conclude graphs planarity directly embedding it in the plane in  cases when the graph is small or built up from 3-connected components in a certain way.

\begin{thm}\label{mth} For $\forall G \in FP$ and $\forall e \not\in E(G)$ $G+e$ is planar.
\end{thm}

Let us prove this theorem in several steps: firstly, enumerating by several theorems all possible graphs belonging to $FP$ and thereafter, by direct check of each graph (or class of graphs) stating the assumption of the theorem.

Let us denote by $\xi$ (see fig.~\ref{fig1}) a particular graph $K_{2,3}$ with an extra hanging edge added to the vertex [$s$ with hanging end $t$] of degree $2$. Let vertices in $K_{2,3}$ of degree $3$ be denoted $x$ and $y$. Let the remaining vertices of degree $2$ be $u$ and $v$.

\placedrawing{FIGURE1.LP}{Graph $\xi$}{fig1}

Let us denote by $m_i  \space (i>0)$ (see fig.~\ref{fig2}) a graph, that actually is a multiedge of degree $i$ with $i-1$  (elementary) subdivided edges (naming it  {\em $i$-multiedge}), e. g. $m_1 \cong K_2, m_2 \cong C_3, m_3 \cong K_4^-$.

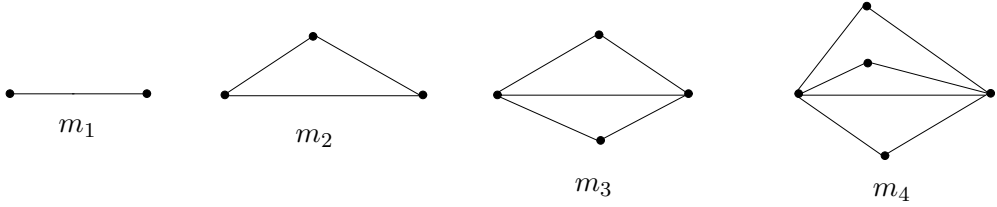
\placedrawing{FIGURE2.LP}{Graph $m_i, i=1,...,4$}{fig2}

\begin{thm}\label{xi}{\em [Subgraph $\xi$ theorem]} \ \ \
If $G$ in $FP$ is 3-connected, then $\xi$ is not its minor.
\end{thm}

Let us first prove a lemma.

\begin{lem}\label{ls} If $G$ in $FP$ is 3-connected, then $m_4$ is not its minor.
\end{lem}

\begin{proof}  Let us assume, that $G$ is 3-connected, has no one reduced
Kuratowski minor, but has  4-multiedge as its minor. But, let us
note, that $m_4$ as a minor is equivalent to $K_5^-$ minus two
incident edges at a vertex of degree four. Further, because of
3-connectivity, these absent edges should be recompensed by a
chain [, uniting two vertices of degree two and going through the
third one and avoiding vertices of degree three (condition of
3-connectivity)]  (see fig.~\ref{fig3}). Thus, existence of
4-multiedge implies existence of   $K_5^-$ too. \\
\end{proof}

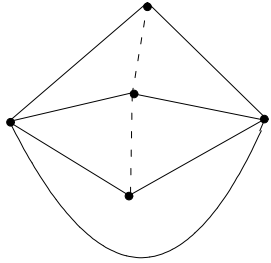
\placedrawing{FIGURE3.LP}{Graph $K_5^-$ with two [dashed] eliminated edges
is equal to $m_4$.}{fig3}

\begin{proof}
{\em  [Proof of the subgraph $\xi$ theorem]} \ \ \ We can not
unite $t$ with any vertex outside the chain $x..s..y$,  without
giving $K_{3,3}^-$, nor unite $t$ with $x$ or $y$, because uniting
$t$ with, say, $y$ and contracting  $x..s$, we get $m_4$.
Furthermore, we can not unite $t$ with  vertices inside the chain
$x .. s .. y$, because contracting the subchains of this chain
from ends until the touch vertex and  $s$ we get $m_4$.  Thus $G$
can not have any minor isomorphic to  $\xi$.
\end{proof}

The fact that $m_4$ is forbidden for graphs in $FP$ can be formulated in the following assertion.

\begin{cor}\label{3ch}{\em [3-chain corollary]}
Let $G$ be 3-connected in $FP$. Then $G$ is isomorphic to $K_n$, $n<5$ or every pair of vertices are joined by $3$ disjoined chains that contain all vertices of the graph and the remaining edges join inner vertices of different chains.
\end{cor}

Still, we need one more theorem that would help us to determine, which graphs belong to $FP$.

\begin{thm} For every 3-connected $G \in FP$ there exists an edge $e$, that $G-e$ is outer planar.
\end{thm}

\begin{proof}  Let us assume $G$ different from $K_n, n<5$ and the theorem is not right, i.e. $G -e$ is not outer planar. Because of  3-chain corollary and 3-connectivity condition, arbitrary pair of vertices $s$ and $t$ are joined by just three chains, where all vertices are positioned on these chains. By the incorrectness assumption every of these chains contain at least one inner vertex, otherwise it should be outer planar. Let us denote these chains $s..x..t$, $s..y..t$ and $s..z..t$. Then, by the same arguments $x$ and $y$ join similar chains too. It is possible, supposing that all inner vertices of $s..z..t$ now are on $x..y$ which avoids $s, t$. But the same argument must be right also for a pair, say, $x$ and $z$. It is impossible without giving $K_5^-$.
\end{proof}

Now we are ready to enumerate 3-connected graphs belonging to $FP$.

\begin{thm}\label{pw}{\em [Prism- and wheel-graph theorem]} \ \ \
The only properly 3-connected graphs belonging to $FP$ are the prism-graph [$\overline{C_6}$] and the wheel-graph [$W_k (k>2)$].
\end{thm}

\begin{proof} Let us assume $G$ different from $K_n, n<4$.
Let us choose the edge $e=s$ (joining vertices $s$ and $t$) that $G-e$ is outer planar.
Then two chains $s..t$ contain all other vertices of the graph $G$.
Let $l$ be the length of the shortest of these chains.
 Case l=1 is not  possible.

For l=2 all cases with the number of inner vertices on the other chain $i>0$ are possible, giving graphs $W_k$ $(k=i+2>3) [wheel \space graph] $.

Let the length of both chains be $3$.  This gives a possible graph $\overline{C_6} [prism-graph] $.

Let both chains be longer than $2$ excluding both being equal to $3$.  Let the chains be $s .. x_1 .. x_2 .. t$ and $s .. y_1 .. y_2..y_3 .. t$. If we join $x_2$ with $y_1$ or $y_2$ then $x_1$ joined with $y_3$ would give $K_{3,3}^-$. By symmetry all other cases are excluded too.
\end{proof}

Up to now, we have considered the cases of  3-connective graphs in $FP$. Further, let us consider other cases and let us state, which edges in the 3-connected graphs eventually can be subdivided and which not in order to get different from 3-connected members of $FP$. Surely, by this reasoning we must get all non 3-connected graphs \cite{ze 90}, because the edges that can be subdivided are just those [and only those], that can become virtual edges, when the graph is divided into 3-connected components.

\begin{thm}{\em [Prism graph edge-subdivision theorem]}
The edges of the triangles in the prism-graph are the only edges that can not be subdivided to get new graphs belonging to $FP$.
\end{thm}

\begin{proof} Putting a new vertex on an edge of a triangle of the
prism-graph immediately gives $K_{3,3}^-$  as a minor. See
fig.~\ref{fig4}. [Names of the vertices in $K_{3,3}$ could be seen
in the fig.~\ref{fig6}]

Putting a new vertex (or  new vertices) on an edge (or edges)  that does not belong to triangle does not give  $K_{3,3}^-$. [There does not exist a cycle with two non elementary bridges.]
\end{proof}

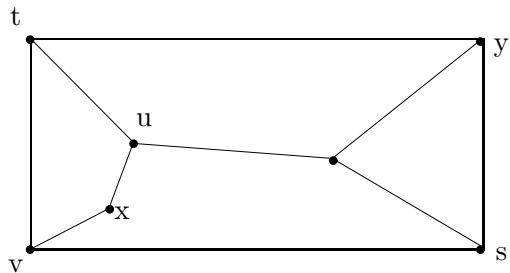
\placedrawing{FIGURE4.LP}{Prism graph with the triangle-edge subdivided, thus giving a minor $K_{3,3}^-$.}{fig4}

\begin{thm}{\em [Wheel graph edge-subdivision theorem]}
The spike edges in the prism-graph are the only edges that can not be subdivided to get new graphs belonging to $FP$.
\end{thm}

\begin{proof}  Putting a new vertex on a spike edge of the
wheel-graph gives immediately $K_{3,3}^-$ as a minor. See
fig.~\ref{fig5}.

Putting a new vertex on a rim-edge does not give  $K_{3,3}^-$. [Union of the new vertex with the center by an edge gives a wheel graph of a higher degree.]
\end{proof}

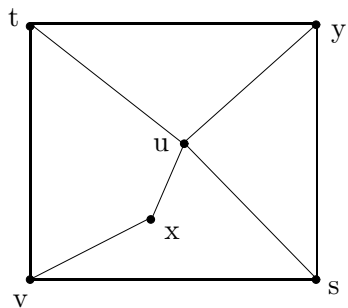
\placedrawing{FIGURE5.LP}{Wheel-graph with the spike-edge subdivided, thus giving a
minor $K_{3,3}^-$.}{fig5}

\begin{thm}{\em [Tetrahedron edge-subdivision theorem]}
Two edges of $K_4$ which subdivided gives  $K_{3,3}^-$ as a minor can not in the same time be subdivided to get new graphs belonging to $FP$.
\end{thm}

\begin{proof}\  Trivially. See fig.~\ref{fig6}.
\end{proof}

We are now ready to specify all the class of graphs $FP$ by enumerating all possible graphs in it. In fact, we name all possible 3-connected graphs in $FP$ with additionally telling which edges in them might become virtual as if the graphs that are not 3-connected would be divided into 3-connected components.

Dealing with the 3-connected components, we must admit , that they are in general multigraphs \cite{ze 90}.

\begin{cor}\label{lc} Graphs or their 3-connected components
that belong to $FP$ are \cite{ze 90}:

\rule{50pt}{0pt} 0) $C_n$ or $m_n, n>2$ with all edges possibly being virtual edges;

\rule{50pt}{0pt} 1) $W_3$ with spike edges possibly being virtual edges;

\rule{50pt}{0pt} 2) $W_k, k>2$ with rim edges possibly being virtual edges;

\rule{50pt}{0pt} 3) $\overline{C_6}$ with possible virtual edges not belonging to triangles.

\end{cor}

\begin{proof}  Dividing the graph into 3-connected components, possible
virtual edges can be only these edges which can eventually be
subdivided, to give possible new members of $FP$.
\end{proof}

\begin{proof}{\em Completion of the proof of theorem \ref{mth}} \ \
\ Now, it can be immediately checked, that adding an edge to the
properly 3-connected graphs of $FP$, i.e. prism-graph and
wheel-graph, can not give a nonplanar graph. This does not need
use of Kuratowski theorem because we infer planarity from direct
implementation in the plane.

Further, looking through all cases of corollary \ref{lc}, immediately can be checked, that subdividing edges in the mentioned graphs, as it is allowed by the 3 last theorems, and adding an extra edge, can not give a  graph, that is not embeddable in the plane.
\end{proof}

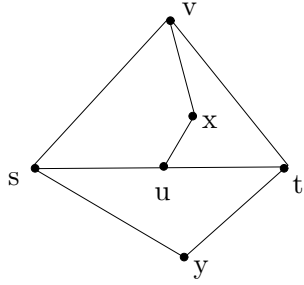
\placedrawing{FIGURE6.LP}{Tetrahedron-graph with two edges subdivided equals to $K_{3,3}^-$.}{fig6}

Now the theorem is proved, saying that adding an edge to $G$ from $FP$ always gives a planar graph. We have proved that $FP$ is a subset of the class of free-planar graphs. Let $Planar$ be class of planar graphs. The result of theorem  \ref{mth} can be expressed in the following lemma.

\begin{lem}\label{lx} $FP \subseteq Free(Planar)$.
\end{lem}

Furthermore, we want to show that these sets in fact coincide. For this purpose, the following lemma is useful.

\begin{lem} $K_5^-, K_{3,3}^- \in F(Free(Planar))$.
\end{lem}

\begin{proof}  It is easy to see, that $K_5^-, K_{3,3}^-$ are
forbidden in $Free(Planar)$ --- addition of an appropriate edge
gives a nonplanar graph. Further, the corresponding elimination of
an edge in both graphs $K_5^-$ and $K_{3,3}^-$, gives four
possibilities for free planar graphs which are shown in
fig.~\ref{fig7}. The corresponding vertex split gives two
non-trivial possibilities [see fig.~\ref{fig7a}].
\end{proof}

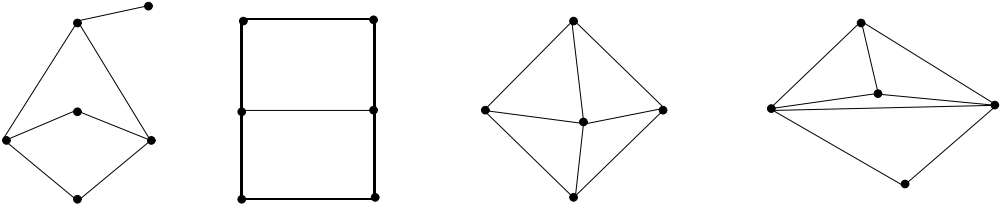
\placedrawing{FIGURE7.LP}{$ K_{3,3}^-$ and $ K_5^-$ without an edge give four non isomorphic graphs}{fig7}

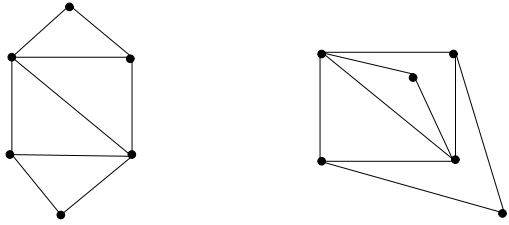
\placedrawing{FIGURE7A.LP}{$ K_5^-$ with a split vertex gives two nonisomorphic graphs}{fig7a}

Further, from two facts, $F(FP) = \{K_5^-, K_{3,3}^-\}$ and
$F(Free(Planar))$ is equal to \\  $ \{K_5^-, K_{3,3}^-, ...something \}$, there follows, that  $ Free(Planar) \subseteq  FP $.
Now, together with lemma \ref{lx} we might formulate, what may be called, the Kuratowski theorem for free-planar graphs.

\begin{thm}[Kuratowski-like theorem for free planar graphs]\label{Kfpg} \ \ \ \ \
   $$F(Free(Planar))= \{K_5^-, K_{3,3}^-\}.  $$
\end{thm}

In fact, as we have already seen in the beginning, this theorem would be easy got using
both traditional Kuratowski theorem and Kratochv\'{\i}l's theorem \cite{kra 94}, but
now we did this proof without  the use of these theorems.

Let us prove Kuratowski theorem from its weaker version, i.e. from
this Kuratowski-like theorem that we have just  proven.

\begin{thm}[Kuratowski theorem-version 1]\label{dkth1}
        $$F(Planar)= \{K_5, K_{3,3}\}.$$
\end{thm}

Let us first prove a lemma.

\begin{lem} Let $H$ be critical non-planar minor. Then $H$ minus
two arbitrary edges is free-planar.
\end{lem}

\begin{proof}
Let $H$ be minimal non-planar minor distinct from Kuratowski
minors and besides let us assume that it is not free-planar after
deleting some two edges from it. Let us assume these edges be $e$
and $f$. Then there must be an edge $h$ so that $H-e-f+h$ is
non-planar. Then 1) $H-e+h[=H']$ is non-planar; 2) $H'$ minus some
non-empty set of edges is critically non-planar [$=H"$] [and
because of minimality of $H$, $H"$ should be equal to one of the
Kuratowski minors]; 3)$H"-e$ is planar graph such that with $h$
becomes non-planar. Let us imagine in the place of $H"$ be some of
Kuratowski graphs. Then there must be some non-edge $h$ such that
Kuratowski graph without arbitrary edge plus $h$ becomes
non-planar. It is not possible for Kuratowski graph[For $K_5$ it
is trivially, for $K_{3,3}$ after some simple consideration].
Contradiction.
\end{proof}

\begin{proof}[Proof of the Kuratowski theorem]
Let us assume that there is some non-planar minor distinct from
Kuratowski minors. It must be free-planar after reduction of two
edges. Let after removing edge $i$ from $H$ reduced Kuratowski
graph $K_i$ be left undestroyed. Let us choose the next edge $j$
from $K_i$ and after this $K_j$ be left undestroyed. Then after
removing both edges $i$ and $j$ graph must be free planar, i.e.
both $K_i$ and $K_j$ should be destroyed. Followingly, $i$ must
belong to the edges of $K_j$. Let us choose $i$ and $j$ from
r.K.m. $K_{ij}$, where $i$ leaves $K_i$ undestroyed and $j$ --
correspondingly $K_j$. Then, deleting $i$ and $j$ all three
r.K.m's should disappear, but as a consequence edge sets of $K_i$
and $K_j$ must intersect at least in a subset of two edges,  $i$
and $j$. At least two edges are there that do not belong to this
intersection, i.e. $l_i$ from $K_i$ and $l_j$ from $K_j$.
Eliminating edges $l_i$ and $l_j$ all r.K.m's should disappear,
but $K_{ij}$ is left untouched, thus we have come to
contradiction.
\end{proof}

Further we give a proof of the Kuratowski theorem for free-planar
graphs, which serves as a proof for Kuratowski theorem for all
class of planar graphs too.

\begin{thm}[Kuratowski theorem-version 2]\label{dkth}
        $$F(Planar)= \{K_5, K_{3,3}\}.$$
\end{thm}

\begin{proof}

Without loss of generality we suppose that graph $G$ is two-connected.

Let us assume that theorem is not right and $G$ is not free planar and it does not contain reduced Kuratowski  minors. Then there is a cycle $C$ with two vertices $x, y$ on it and at least two bridges $B_x$ and $B_y$ that screen $x$ from $y$ on $C$ and either they are not placeable on one side against $C$ or they are connected [i.e. not placeable together] with an alternating [i.e. on one and other side of $C$] sequence [$B_1,...,B_{2k}, k>0$] of non-screening [$x$ from $y$] bridges. Finding of reduced Kuratowski minors would reprove the incorrectness assumption.

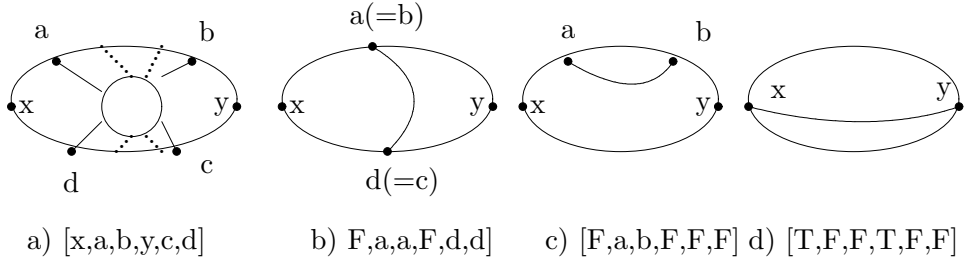
\placedrawing{FIGUREB.LP}{A bridge $[x,a,b,y,c,d]$ with respect to a cycle with two distinguished vertices $x$ and $y$: a) a bridge in general; b)a trivial screening bridge $[F,a,a,F,d,d]$; c) a trivial non-screening bridge $[F,a,b, F,F,F]$; d) edge $x, y$ as a bridge with respect to $C$ $[T,F,F,T,F,F]$.}{figureBlp}

Let us describe bridge with sextet $[x, a, b, y, c, d]$, where values of it are either vertices on the cycle $C$ or logical values $T (=true)$ or $F (=false)$ [see fig. \ref{figureBlp}]:

1) in place of $x (y)$ stands $T$ if $x (y)$ is a leg[i.e. touch vertex to $C$] of the bridge, otherwise $F$;

2) $a (c)$ is nearest leg clockwise from $x (y)$, if different from
$y (x)$, otherwise $F$;

3) $b (d)$ is nearest leg anticlockwise from $y (x)$, if different from $x (y)$, otherwise $F$;

The screening condition of bridge $[x,a,b,y,c,d]$ of $x$ from $y$ on $C$ is -- values $a, b, c, d$ are not $F$. Non-screening bridges $B_i, [0 < i \leq 2k]$ are of a form $[x, a, b, y, F, F]$ or $[x, F, F, y, c, d]$ in general, but taken together with $B_x$ and $B_y$ in place of $x$ and $y$ should stand $F$.

There are three simple [$k = 0$] cases and one non-simple [$k > 1$] case to be considered:

1) In one of bridges, say $B_x$, both in $x$ and $y$ stand $T$. In this case $K_5^-$ arises even when $B_y$ is simple: $[F, a, a, F, c, c]$.

2) If $x (y)$ is $T$ in both $B_x$ and $B_y$, then $K_5^-$ arises too: simplest case -- both bridges are $[T, a, a, F, d, d]$ with minimal number of edges giving $K_5^-$ with subdivided edge [by $y(x)$ on $C$].

3) Bridges $B_x$ of form $[x, a_x, b_x, F, c_x, d_x]$ and $B_y$ of form $[F, a_y, b_y, y, c_y, d_y]$ [where in $x$ and (or) $y$ may stand $F$] are not placeable on one side  when legs' non intersecting condition --existence of two followingly specified  pathes $$x..a_1.b_1.a_2.b_2..y,$$
$$x..d_1.c_1.d_2.c_2..y.$$ -- is not hold.

When this condition is not true, easy checkable $K_{3,3}^-$ arises.

4) Similarly as in case 3 bridges $B_x$ and $B_y$ can not be placed on one side of $C$, if alternating sequence of bridges, say of form,  $[F, a_i, b_i, F, F, F]$ [$0<i \leq 2k$]  join them when the condition -- existence of path
$$x.a_1..b_x. a_2..b_1. a_3.. \ \ ... \ \ . a_{2k}.. b_{2k-1} .a_y..b_{2k}.y$$
-- is hold.

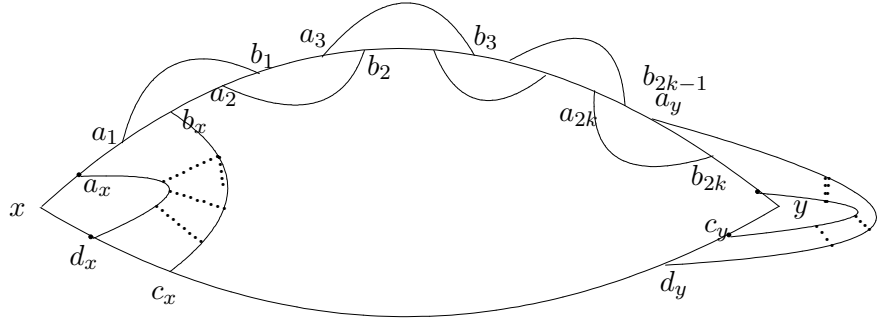
\placedrawing{KURA.LP}{Case 4 in the proof of Kuratowski theorem}{kuralp}

When the bridges joining condition is true, $K_{3,3}^-$ arises [see fig. \ref{kuralp}]:

1) cycle
$$a_y..d_y..c_x..b_x.a_2.b_2.a_4.  \ \ ... \ \ .b_{2k-2}.a_{2k}..b_{2k-1}.a_y;$$

2) a chain through $x$:
$$c_x..x.a_1..b_1.a_3..  \ \ ... \ \ .a_{2k-1}..b_{2k-1};$$

3) a chain through $y$:
$$d_y..y.b_{2k}..a_{2k}  \ \ .$$

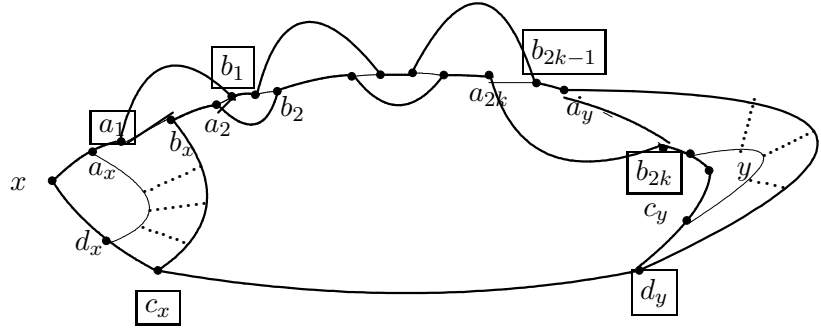
\placedrawing{KURA1.LP}{Minor $K_{3,3}$ bold: 1) cycle avoiding $x$ and $y$, 2) chain through $x$ outside and 3) chain through $y$ inside}{kura1lp}

It can be seen from fig. \ref{kuralp}[and fig. \ref{kura1lp} with $K_{3,3}^-$ bold] that both the cycle of supposed $K_{3,3}^-$ and the chain through $y$ goes through even vertices belonging to, say, inner bridges of joining sequence of bridges. The chain through $x$ goes through odd vertices, i.e. outer bridges of the sequence of joining bridges.

Thus $G$ must have reduced Kuratowski graphs as its minors and $G
+ xy$ correspondingly -- Kuratowski graph as its minor. This
completes the proof of the Kuratowski theorem .

\end{proof}

It is easy to see that case $3$ in the last proof is not necessary, i.e. it is equal
to case $4$ with $k=0$.


\end{document}